\documentclass[12pt,reqno]{amsart}
\usepackage{indentfirst}
\usepackage{latexsym,bm}
\usepackage{xypic}
\usepackage{wasysym}
\usepackage{indentfirst}
\usepackage{amssymb}
\usepackage{dsfont}
\usepackage{amsthm}

\usepackage{yfonts}
\usepackage{amsmath,amsfonts}
\usepackage{enumerate}
\begin{document}
\newcommand{\fr}[2]{\frac{\;#1\;}{\;#2\;}}
\newtheorem{theorem}{Theorem}[section]
\newtheorem{lemma}{Lemma}[section]
\newtheorem{proposition}{Proposition}[section]
\newtheorem{corollary}{Corollary}[section]
\newtheorem{remark}{Remark}[section]
\newtheorem{definition}{Definition}[section]
\newtheorem{notation}{Notation}[section]
\newtheorem{example}{Example}[section]

\newcommand{\Hom}{\mathrm{Hom}\,}

\title[The Exponent of Non-Cosemisimple Hopf Algebras]{On the Exponent of Finite-Dimensional Non-Cosemisimple Hopf Algebras}

\author[K. Li and S. Zhu]{Kangqiao Li and Shenglin Zhu \\
School of Mathematical Sciences, Fudan University, Shanghai, China}  

\thanks{\textbf{CONTACT} Shenglin Zhu
mazhusl@fudan.edu.cn   
School of Mathematical Sciences, Fudan University, No. 220, Handan Road, Shanghai, China.}   

\date{}

\begin{abstract}
In 1999, Y. Kashina introduced the  exponent of a Hopf algebra. In this paper, we prove that the exponent of a finite dimensional non-cosemisimple Hopf algebra with Chevalley property in characteristic 0 is infinite, and the exponent of a finite dimensional non-cosemisimple pointed Hopf algebra in positive characteristic is finite.
\end{abstract}

\maketitle


\textbf{2000 MATHEMATICS SUBJECT CLASSIFICATION}: 16W30

\section{Introduction}

Let $(H,m,u,\Delta,\varepsilon,S)$ be a Hopf algebra over a field $\Bbbk$. We use Sweedler's  notation $\Delta(h)=\sum\limits_h h_{(1)}\otimes h_{(2)}$ for $h\in H$.

Y. Kashina \cite{Kashina 2000} defined the $n$th Hopf power as $[n]=m_n\circ\Delta_n:H\rightarrow H$, where $\Delta_n:H\rightarrow H^{\otimes n},h\mapsto \sum\limits_h h_{(1)}\otimes h_{(2)}\otimes \cdots \otimes h_{(n)}$ and $m_n:H^{\otimes n}\rightarrow H,h_1\otimes h_2\otimes \cdots \otimes h_n\mapsto h_1h_2\cdots h_n$, for all $h,\ h_1,h_2,\cdots,h_n\in H$. The exponent of $H$ is the least positive integer $n$ such that $[n]=u\circ \varepsilon$. We denote the exponent of $H$ by $\exp(H)$ if such an $n$ exists, otherwise we write $\exp(H)=\infty$. For a semisimple and cosemisimple Hopf algebra $H$, Y. Kashina \cite{Kashina 1999} conjectured that ${\rm exp}(H)$ is finite and divides the dimension of $H$.

In 1999, P. Etingof and S. Gelaki \cite{E-G 1999} used the notion of exponent of $H$ to stand for the least positive integer $n$ such that $ \sum\limits_h h_{(1)}\, S^{-2}\left( h_{(2)}\right)\, \cdots  \,S^{-2n+2}\left( h_{(n)}\right) =\varepsilon(h)1_H$ for all $h\in H$. For a semisimple and cosemisimple Hopf algebra $H$, the two notions of exponent are the same, and the exponent of $H$ divides  $\left(\dim_{\Bbbk} H\right)^3$. In 2006, Y. Kashina, Y. Sommerh\"{a}user, Y. Zhu \cite{K-S-Z 2006} proved that any prime number dividing $\dim_{\Bbbk} H$ also divides $\exp(H)$.

In this paper, we study the exponent of non-cosemisimple Hopf algebras. We prove in Section 4 that over a field of characteristic $0$, the exponent of a non-cosemisimple Hopf algebra with Chevalley property is $\infty$. When the characteristic of the base field is positive, we prove in Section 5 that the exponent of a finite-dimensional non-cosemisimple pointed Hopf algebra $H$ is  finite.

\section{Coradical Orthonormal Idempotents in $H^\ast$}
Throughout this paper, we let ${\Bbbk}$ be a field, $(H,\Delta,\varepsilon)$ be a coalgebra over ${\Bbbk}$, and $H^\ast$ be the dual (algebra) of $H$. Denote the the coradical filtration of $H$ by $\{H_n\}_{n=0}^\infty$, and the set of all simple subcoalgebras of $H$ by $\mathcal{S}$, then we have $H_0=\bigoplus\limits_{C\in \mathcal{S}} C$.

\begin{definition}[\protect{\cite[Definition 3.5.12]{Radford 2012}}]
Let $H$ be a coalgebra. An orthonormal family of idempotents of $H^\ast$ is a family of non-zero idempotents $\{e_i\}_{i\in \mathcal{I}}$ of $H^\ast$ such that $e_i e_j=0$ whenever $i,j\in\mathcal{I}$ are different and $\sum\limits_{i\in \mathcal{I}} e_i=\varepsilon$.
\end{definition}

\begin{remark}
Assume that a family of idempotents $\{e_i\}_{i\in \mathcal{I}}$ of $H^\ast$ satisfying the condition that $e_i e_j=0$ whenever $i,j\in\mathcal{I}$ are different. Then $\sum\limits_{i\in \mathcal{I}} e_i$ is a well-defined idempotent since $H$ is locally finite by the fundamental theorem of coalgebras.
\end{remark}

D.E. Radford affirmed the existence of an orthonormal family of idempotents $\{e_C\}_{C\in \mathcal{S}}\subseteq H^\ast$ on every coalgebra $H$ such that the natural actions (both left and right) of $H^*$ on $H_0$  by each $e_C$ is exactly the ${\Bbbk}$-projection from $H_0$ onto $C$.

\begin{proposition}[\protect{\cite[Definition 3.5.15]{Radford 2012}}]\label{1}
Let $H$ be a coalgebra. Let $H_0=\bigoplus\limits_{C\in \mathcal{S}} C$ be the coradical of $H$. Then there exists an orthonormal family of idempotents $\{e_C\}_{C\in\mathcal{S}}$ of $H^\ast$ such that $e_C|_D=\delta_{C,D}\varepsilon|_D$, where $\delta$ is the Kronecker delta.
\end{proposition}

In this paper, an orthonormal family of idempotents $\{e_C\}_{C\in\mathcal{S}}\subseteq H^\ast$ satisfying the property in the previous proposition will be called \emph{(a family of) coradical orthonormal idempotents} of $H^\ast$. It's not unique in general. For any family of coradical orthonormal idempotents $\{e_C\}_{C\in\mathcal{S}}$, we use for convenience the notations below
$${^C}h=h\leftharpoonup e_C,h^D=e_D\rightharpoonup h,{^C}h^D=e_D\rightharpoonup h\leftharpoonup e_C,\ \forall~h\in H,\forall~C,D \in \mathcal{S}.$$
Specially if $C={\Bbbk} g$ is pointed, we write ${}^gh$ as ${}^{C}h$, $h^g$ as $h^{C}$, and  $V^g=V^{{\Bbbk} g}=e_C\rightharpoonup V$, et..

\begin{proposition}\label{2}
Let $H$ be a coalgebra and $\{e_C\}_{C\in\mathcal{S}}\subseteq H^\ast$ be a family of coradical orthonormal idempotents. Then for all $C,D\in\mathcal{S}$, we have
\begin{enumerate}
\item ${}^C H_0{}^D=\delta_{C,D}C$.
\item ${}^C H_1{}^D \subseteq \Delta^{-1}(C\otimes {}^C H_1{}^D + {}^C H_1{}^D \otimes D)$.
\item ${}^C H_1{}^D \subseteq {\rm Ker}~\varepsilon$ if $C\neq D$.
\item For any ${\Bbbk}$-subspace $V\subseteq H$, we have the direct-sum decomposition
    $V=\bigoplus\limits_{C\in\mathcal{S}}{}^C V=\bigoplus\limits_{C\in\mathcal{S}} V^D=\bigoplus\limits_{C,D\in\mathcal{S}} {}^C V^D$.
\end{enumerate}
\end{proposition}

\section{Multiplicative Matrices and Primitive Matrices}

In 1988, Y. Manin \cite{Manin 1988} introduced the concept ``multiplicative matrix'' for  quantum group constructions. We first recall Manin's notations on matrices over ${\Bbbk}$-vector spaces. For positive integers $r$ and $s$, we use $\mathcal{M}_{r\times s}(V)$ to denote the set of all $r\times s$ matrices over a vector space $V$. If $r=s$, we write $\mathcal{M}_r(V)=\mathcal{M}_{r\times r}(V)$.

\begin{notation}[\protect{\cite[Section 2.6]{Manin 1988}}]
Let $V,W$ be ${\Bbbk}$-vector spaces. Define a bilinear map
$$\widetilde{\otimes}:\mathcal{M}_{r\times s}(V)\otimes\mathcal{M}_{s\times t}(W)\rightarrow\mathcal{M}_{r\times t}(V\otimes W),
(v_{ij})\otimes(w_{kl})\mapsto \left( \sum\limits_{k=1}^s v_{ik} \otimes w_{kj} \right).$$
\end{notation}

\begin{remark}
$\widetilde{\otimes}$ is not the usual tensor product of $\mathcal{M}_{r\times s}(V)$ and $\mathcal{M}_{s\times t}(W)$. However, the associative law
$\widetilde{\otimes}\circ(\widetilde{\otimes}\otimes id)=\widetilde{\otimes}\circ(id\otimes\widetilde{\otimes})$ on
$\mathcal{M}_{r\times s}(U)\otimes\mathcal{M}_{s\times t}(V)\otimes\mathcal{M}_{t\times u}(W)$ holds.
\end{remark}

In order to calculate counits and coproducts of all the entries in a matrix over a coalgebra, we use the following notations.

\begin{notation}[\protect{\cite[Section 2.6]{Manin 1988}}] Let $V,\ W$ be any ${\Bbbk}$-vector spaces,and let $f\in  {\Hom}_{{\Bbbk}}(V,W)$, we use the same symbol $f$ to denote the linear map
$f:\mathcal{M}_{r\times s}(V)\rightarrow \mathcal{M}_{r\times s}(W),(v_{ij})\mapsto \left(  f(v_{ij}) \right)$. Specifically,
\begin{enumerate}
\item If $(H,\Delta,\varepsilon)$ is a coalgebra, then
    \begin{eqnarray*}
    \Delta&:&\mathcal{M}_{r\times s}(H)\rightarrow\mathcal{M}_{r\times s}(H\otimes H),(h_{ij})\mapsto\left(\Delta(h_{ij})\right) \\
    \varepsilon&:&\mathcal{M}_{r\times s}(H)\rightarrow\mathcal{M}_{r\times s}({\Bbbk}),(h_{ij})\mapsto\left(\varepsilon(h_{ij})\right).
    \end{eqnarray*}
\item If $(H,m,u)$ is an algebra, then
    $$m:\mathcal{M}_{r\times s}(H\otimes H)\rightarrow\mathcal{M}_{r\times s}(H),(h_{ij}\otimes h_{ij}{}^\prime)\mapsto(h_{ij} h_{ij}{}^\prime).$$
\item If $H$ is Hopf algebra with antipode $S$, then
    $$S:\mathcal{M}_{r\times s}(H)\rightarrow\mathcal{M}_{r\times s}(H),(h_{ij})\mapsto\left(S(h_{ij})\right).$$
\end{enumerate}
\end{notation}

The following proposition is immediate.

\begin{proposition}\label{6}
\begin{enumerate}
\item Let $(H,\Delta,\varepsilon)$ be a coalgebra, then the equalities $(\Delta\otimes id)\circ \Delta=(id\otimes\Delta)\circ \Delta$ and $(\varepsilon\otimes id)\circ \Delta=(id\otimes\varepsilon)\circ \Delta$ hold on $\mathcal{M}_{r\times s}(H)$.
\item Let $(H,m,u)$ be an algebra, then $m\circ\widetilde{\otimes}$ on $\mathcal{M}_{r\times s}(H)\otimes \mathcal{M}_{s\times t}(H)$ is exactly the matrix multiplication $\mathcal{M}_{r\times s}(H)\otimes \mathcal{M}_{s\times t}(H)\to \mathcal{M}_{r\times t}(H)$.
\end{enumerate}
\end{proposition}

A ``multiplicative matrix'' over a coalgebra has similar properties to a group-like element.

\begin{definition}[\protect{\cite[Section 2.6]{Manin 1988}}]
Let $(H,\Delta,\varepsilon)$ be a coalgebra and $r$ be a positive integer. Let $I_r$ denote the identity matrix of order $r$ over ${\Bbbk}$. A matrix $\mathcal{G}\in \mathcal{M}_r(H)$ is called a multiplicative matrix over $H$ if $\Delta(\mathcal{G})=\mathcal{G}~\widetilde{\otimes}~\mathcal{G}$ and $\varepsilon(\mathcal{G})=I_r$.
\end{definition}

We also call a square matrix $\mathcal{G}$ over a ${\Bbbk}$-vector space $H$ with linear maps $\Delta:H\rightarrow H\otimes H$ and $\varepsilon:H\rightarrow{\Bbbk}$ \emph{multiplicative} whenever $\Delta(\mathcal{G})=\mathcal{G}~\widetilde{\otimes}~\mathcal{G}$ and $\varepsilon(\mathcal{G})=I$ hold, even if $(H,\Delta,\varepsilon)$ is not a coalgebra.

\begin{example}
Let $C$ be a simple coalgebra over an algebraically closed field ${\Bbbk}$. Then $\dim_{\Bbbk} C=r^2$ for some natural integer $r$, and there exists a basis $\{c_{ij} \mid 1\leq i,j\leq r\}$ such that  $\Delta(c_{ij})=\sum\limits_{k=1}^r c_{ik}\otimes c_{kj}$ and $\varepsilon(c_{ij})=\delta_{ij}$. Then the $r\times r$ matrix $\mathcal{C}=(c_{ij})$ is a multiplicative matrix over $C$.
\end{example}

A multiplicative matrix $\mathcal{C}$ over a simple coalgebra $C$ is called a \emph{basic multiplicative matrix} of $C$ if all the entries of $\mathcal{C}$ form a basis of $C$. It is not unique in general.

Now we give a description of the left integral of a Hopf algebra with basic multiplicative matrices.

\begin{proposition}[cf. \protect{\cite[Proposition 9]{L-S 1969}}]\label{8}
Let $H$ be a finite-dimensional involutory Hopf algebra. Define $\Lambda_0\in H$ by $\langle h^\ast,\Lambda_0\rangle={\rm Tr}(L(h^\ast))$ for all $h^\ast\in H^\ast$, where ${\rm Tr}(L(h^\ast))$ is the trace of the left multiplication
$$
L(h^\ast): H^\ast\rightarrow H^\ast,k^\ast \mapsto h^\ast k^\ast.
$$
Then
\begin{enumerate}
\item $\Lambda_0$ is a left integral of $H$.
\item Let $\{h_i\mid 1\leq i\leq n\}$ be a basis of $H$ ,and let $\{h_i^\ast \mid 1\leq i\leq n\}$ be the dual basis of $H^\ast$. Then $\Lambda_0=\sum\limits_{i=1}^n h_i^\ast\rightharpoonup h_i$.
\item Assume that ${\Bbbk}$ is algebraically closed and $H$ is cosemisimple. For each  simple subcoalgebra $D\in \mathcal{S}$ with dimension $r_D{}^2$, choose a basic multiplicative matrix $(d_{ij}^D)_{i,j=1}^{r_D}$, and denote $t_D=\sum\limits_{i=1}^{r_D} d_{ii}^D$. Then we have $\Lambda_0=1+\sum\limits_{D\in\mathcal{S}\setminus\{{\Bbbk} 1\}} r_D t_D$.
\end{enumerate}
\end{proposition}

\begin{proof}
Since ${\rm Tr}(L(h^\ast))=\sum\limits_{i=1}^n \langle h^\ast h_i^\ast , h_i\rangle=\langle h^\ast,\sum\limits_{i=1}^n h_i^\ast \rightharpoonup h_i\rangle$ for all $h^\ast\in H^\ast$, $\Lambda_0=\sum\limits_i h_i^\ast \rightharpoonup h_i$ and (2) holds.

Assume that ${\Bbbk}$ is algebraically closed and $H$ is cosemisimple as a coalgebra. For each simple subcoalgebra $D\in\mathcal{S}$, choose a basic multiplicative matrix $(d_{ij}^D)_{i,j=1}^{r_D}$ of it. Then $\{d_{ij}^D \mid D\in\mathcal{S},1\leq i,j\leq r_D\}$ is a basis of $H$, and let $\{d_{ij}^{D\ast} \mid D\in\mathcal{S},1\leq i,j\leq r_D\}$ be the dual basis of $H^\ast$. An obvious calculation gives $d_{ij}^{D\ast} \rightharpoonup d_{ij}^D=\sum\limits_{k=1}^{r_D} d_{ik}^D\langle d_{ij}^{D\ast} , d_{kj}^D\rangle=d_{ii}^D$. Summing them over, we obtain
$$
\Lambda_0
=\sum\limits_{D\in\mathcal{S}} \sum\limits_{i,j=1}^{r_D} d_{ij}^{D\ast} \rightharpoonup d_{ij}^D
=1+\sum\limits_{D\in\mathcal{S}\setminus\{{\Bbbk} 1\}} r_D t_D
.
$$
\end{proof}

Primitive elements play an important role in the study of pointed Hopf algebras. With multiplicative matrices, we define a similar concept.

\begin{definition}
Let $H$ be a coalgebra. For $C,D\in\mathcal{S}$, let $\mathcal{C}_{r\times r},\mathcal{D}_{s\times s}$ be basic multiplicative matrices of  $C$ and $D$ respectively. A matrix $\mathcal{W}\in\mathcal{M}_{r\times s}(H)$ is called a $\left(\mathcal{C},\mathcal{D}\right)$-primitive matrix if $\Delta(\mathcal{W})=\mathcal{C}~\widetilde{\otimes}~\mathcal{W}+\mathcal{W}~\widetilde{\otimes}~\mathcal{D}$.
\end{definition}

\begin{remark}
In fact, if $\mathcal{W}\in\mathcal{M}_{r\times s}(H)$ is a $\left(\mathcal{C},\mathcal{D}\right)$-primitive matrix, then
$
\left(
\begin{array}{cc}
  \mathcal{C} & \mathcal{W} \\
  0 & \mathcal{D}
\end{array}
\right)
$ is a multiplicative matrix of order $r+s$.
\end{remark}

From now on, we assume that ${\Bbbk}$ is an algebraically closed field, and $H$ is a Hopf algebra over ${\Bbbk}$. We fix a family of coradical orthonormal idempotents $\{e_C\}_{C\in \mathcal{S}}$ of $H^\ast$. For any subspace $V\subseteq H$, we denote $V\cap {\rm Ker}~\varepsilon$ by $V^+$.

Recall in Proposition~\ref{2}~(2) that for any $C,D\in \mathcal{S}$,
$$
\Delta({}^C H_1{}^D) \subseteq C\otimes {}^C H_1{}^D + {}^C H_1{}^D \otimes D.
$$
Thus if $\mathcal{C}=(c_{i^\prime i})_{r\times r},\mathcal{D}=( d_{jj^\prime })_{s\times s}$ be basic multiplicative matrices of $C,D$ respectively. For any $w\in {}^C H_1{}^D$, we can write
$$
\Delta(w)=\sum\limits_{i^\prime ,i=1}^r c_{i^\prime i}\otimes x_i^{(i^\prime )}
         +\sum\limits_{j,j^\prime =1}^s y_j^{(j^\prime )}\otimes d_{jj^\prime } ,
$$
where $x_i^{(i^\prime )}, y_j^{(j^\prime )}\in {}^C H_1{}^D$ for each $1\leq i^\prime ,i\leq r$, and $1\leq j,j^\prime \leq s$. Furthermore, we have the following.

\begin{lemma}\label{3}
Let $H$ be a coalgebra. Then for any $w\in {}^C H_1{}^D$,
\item[\hspace*{2mm}$(1)$] If $C\neq D$, then $x_i^{(i^\prime )}, y_j^{(j^\prime )}\in ({}^C H_1{}^D)^+$ for all $1\leq i^\prime ,i\leq r,\ 1\leq j,j^\prime \leq s$. In other words,
$$
\Delta(w) \in C\otimes({}^C H_1{}^D)^+ +({}^C H_1{}^D)^+ \otimes D .
$$
\item[\hspace*{2mm}$(2)$] If $C=D$, if we choose that $\mathcal{C}=\mathcal{D}$, then
$$
\Delta\left( w-\sum\limits_{i^\prime ,i=1}^r \varepsilon( x_i^{(i^\prime)} + y_{i^\prime}^{(i)} ) c_{i^\prime i} \right) \in C\otimes({}^C H_1{}^C)^+ +({}^C H_1{}^C)^+ \otimes C .
$$
\end{lemma}

\begin{proof}

\begin{enumerate}
\item If $C\neq D$, $\varepsilon({}^C H_1{}^D)=0$  by Proposition~\ref{2} (3), thus the claim holds.

\item If $C=D$ and $\mathcal{C}=\mathcal{D}$,
\begin{eqnarray*}
\Delta\left( w-\sum\limits_{i^\prime ,i=1}^r\varepsilon(x_i^{(i^\prime )}+y_{i^\prime }^{(i)})c_{i^\prime i}\right)
&=& \Delta(w)
  - \sum\limits_{i^\prime ,i=1}^r\varepsilon(x_i^{(i^\prime )}) \Delta(c_{i^\prime i})
  - \sum\limits_{i^\prime ,i=1}^r\varepsilon(y_{i^\prime }^{(i)}) \Delta(c_{i^\prime i})  \\
&=& \Delta(w)
  - \sum\limits_{i^\prime ,i=1}^r\varepsilon(x_i^{(i^\prime )}) \Delta(c_{i^\prime i})
  - \sum\limits_{j,j^\prime =1}^r\varepsilon(y_j^{(j^\prime )}) \Delta(c_{jj^\prime })  \\
&=& \sum\limits_{i^\prime ,i=1}^r c_{i^\prime i}\otimes x_i^{(i^\prime )}
  + \sum\limits_{j,j^\prime =1}^s y_j^{(j^\prime )}\otimes c_{jj^\prime }\\
 && - \sum\limits_{i^\prime ,i,k=1}^r\varepsilon(x_i^{(i^\prime )}) c_{i^\prime k}\otimes c_{ki}
  - \sum\limits_{j,j^\prime ,l=1}^r\varepsilon(y_j^{(j^\prime )}) c_{jl}\otimes c_{lj^\prime }  \\
&=& \sum\limits_{i^\prime ,i=1}^r c_{i^\prime i}\otimes x_i^{(i^\prime )}
  + \sum\limits_{j,j^\prime =1}^s y_j^{(j^\prime )}\otimes c_{jj^\prime }\\
&&
  - \sum\limits_{i^\prime ,k,i=1}^r c_{i^\prime i}\otimes \varepsilon(x_k^{(i^\prime )})c_{ik}
  - \sum\limits_{l,j^\prime ,j=1}^r\varepsilon(y_l^{(j^\prime )}) c_{lj}\otimes c_{jj^\prime }  \\
&=& \sum\limits_{i^\prime ,i=1}^r c_{i^\prime i}\otimes
      \left[ x_i^{(i^\prime )} -\sum\limits_{k=1}^r \varepsilon(x_k^{(i^\prime )})c_{ik} \right]\\
 & &+ \sum\limits_{j,j^\prime =1}^s
      \left[ y_j^{(j^\prime )}-\sum\limits_{l=1}^r\varepsilon(y_l^{(j^\prime )}) c_{lj} \right]
            \otimes c_{jj^\prime }  .
\end{eqnarray*}
It's easy to verify that $$x_i^{(i^\prime )} -\sum\limits_{k=1}^r \varepsilon(x_k^{(i^\prime )})c_{ik},\ y_j^{(j^\prime )}-\sum\limits_{l=1}^r\varepsilon(y_l^{(j^\prime )}) c_{lj} \in ({}^C H_1{}^D)^+$$ for each $1\leq i^\prime ,i\leq r,\ 1\leq j,j^\prime \leq s$.
\end{enumerate}
\end{proof}

\begin{theorem}\label{7}
Let $H$ be a coalgebra, $C,D$ be simple subcoalgebras of $H$ and $\mathcal{C}=(c_{i^\prime i})_{r\times r},\mathcal{D}=( d_{jj^\prime })_{s\times s}$ be respectively basic multiplicative matrices for $C$ and $D$. Then
\begin{enumerate}
\item If $C\neq D$, then for any $w\in {}^C H_1{}^D$, there exist $rs$  $\left(\mathcal{C},\mathcal{D}\right)$-primitive matrices
        $$
        \mathcal{W}^{(i^\prime ,j^\prime )}= \left( w_{ij}^{(i^\prime ,j^\prime )} \right)_{r\times s}~~
        (1\leq i^\prime \leq r,1\leq j^\prime \leq s),
        $$
        such that $w=\sum\limits_{i=1}^r \sum\limits_{j=1}^s w_{ij}^{(i,j)}$;
\item If $C=D$, then for any $w\in {}^C H_1{}^D$, there exist $rs$ $\left(\mathcal{C},\mathcal{D}\right)$-primitive matrices
        $$
        \mathcal{W}^{(i^\prime ,j^\prime )}= \left( w_{ij}^{(i^\prime ,j^\prime )} \right)_{r\times s}~~
        (1\leq i^\prime \leq r,1\leq j^\prime \leq s),
        $$
        such that $w- \sum\limits_{i=1}^r \sum\limits_{j=1}^s w_{ij}^{(i,j)}\in C$.
\end{enumerate}
\end{theorem}

\begin{proof}
\begin{enumerate}
\item
If $C\neq D$,  then by Lemma~\ref{3} (1),
$$
\Delta(w)=\sum\limits_{i^\prime ,i=1}^r c_{i^\prime i}\otimes x_i^{(i^\prime )}
         +\sum\limits_{j,j^\prime =1}^s y_j^{(j^\prime )}\otimes d_{jj^\prime } ,
$$
where $x_i^{(i^\prime )}, y_j^{(j^\prime )}\in ({}^C H_1{}^D)^+$ for all $1\leq i^\prime ,i\leq r,1\leq j,j^\prime \leq s$. Then
\begin{equation}\label{E1}
    w=w\leftharpoonup \varepsilon=\sum\limits_{i=1}^r x_i^{(i)} .
\end{equation}

We write for convenience $\Delta(w)\in \sum\limits_{i^\prime ,i=1}^r c_{i^\prime i}\otimes x_i^{(i^\prime )} + H_1\otimes D$. By coassociativity,
\begin{eqnarray*}
(id \otimes\Delta)\circ\Delta(w)
&  \in &\sum\limits_{i^\prime ,i=1}^r c_{i^\prime i}\otimes \Delta(x_i^{(i^\prime )}) + H_1\otimes D\otimes D  \\
&  \subseteq & \sum\limits_{i^\prime ,i=1}^r c_{i^\prime i}\otimes \Delta(x_i^{(i^\prime )})
          + H_1\otimes H_1\otimes D  ,
\end{eqnarray*}
on the other hand, \begin{eqnarray*}
(\Delta\otimes id)\circ\Delta(w)
&  \in &\sum\limits_{i^\prime ,i,k=1}^r c_{i^\prime k}\otimes c_{ki}\otimes x_i^{(i^\prime )}
    + H_1\otimes H_1\otimes D  \\
&  \subseteq &\sum\limits_{i^\prime ,k,i=1}^r c_{i^\prime i}\otimes c_{ik}\otimes x_k^{(i^\prime )}
          + H_1\otimes H_1\otimes D .
\end{eqnarray*}
Therefore
$
\sum\limits_{i^\prime ,i=1}^r c_{i^\prime i}\otimes
\left[ \Delta(x_i^{(i^\prime )}) -\sum\limits_{k=1}^r c_{ik}\otimes x_k^{(i^\prime )} \right]
\in H_1\otimes H_1\otimes D
$.
 As $\{c_{i^\prime i} \mid 1\leq i^\prime ,i\leq r\}$ are linearly independent, thus for  each $1\leq i^\prime ,i\leq r$,
\begin{eqnarray*}
\Delta(x_i^{(i^\prime )})& -&\sum\limits_{k=1}^r c_{ik}\otimes x_k^{(i^\prime )} \in H_1\otimes D ,  \\
  \Delta(x_i^{(i^\prime )}) &\in& \sum\limits_{k=1}^r c_{ik}\otimes x_k^{(i^\prime )} + H_1\otimes D .
\end{eqnarray*}
Hence we can write
$$
\Delta(x_i^{(i^\prime )})
= \sum\limits_{k=1}^r c_{ik}\otimes x_k^{(i^\prime )}
  + \sum\limits_{j,j^\prime =1}^s w_{ij}^{(i^\prime ,j^\prime )}\otimes d_{jj^\prime } ,
$$ for $w_{ij}^{(i^\prime ,j^\prime )}\in H_1$. It follows that
\begin{equation}\label{E2}
x_i^{(i^\prime )}=\varepsilon\rightharpoonup x_i^{(i^\prime )}=\sum\limits_{j=1}^s w_{ij}^{(i^\prime ,j)} .
\end{equation}
Furthermore by a similarly argument,
\begin{eqnarray*}
(\Delta\otimes id)\circ\Delta(x_i^{(i^\prime )}) &
=& \sum\limits_{k,m=1}^r c_{im}\otimes c_{mk}\otimes x_k^{(i^\prime )}
  + \sum\limits_{j,j^\prime =1}^s \Delta(w_{ij}^{(i^\prime ,j^\prime )})\otimes d_{jj^\prime } , \\
(id \otimes\Delta)\circ\Delta(x_i^{(i^\prime )}) &
=& \sum\limits_{k=1}^r c_{ik}\otimes
    \left( \sum\limits_{m=1}^r c_{km}\otimes x_m^{(i^\prime )}
           + \sum\limits_{j,j^\prime =1}^s w_{kj}^{(i^\prime ,j^\prime )} \otimes d_{jj^\prime } \right)  \\
  & & + \sum\limits_{j,j^\prime ,l=1}^s w_{ij}^{(i^\prime ,j^\prime )} \otimes d_{jl}\otimes d_{lj^\prime }  \\
 &
=& \sum\limits_{k,m=1}^r c_{ik}\otimes c_{km}\otimes x_m^{(i^\prime )}  \\
  & & + \sum\limits_{k=1}^r \sum\limits_{j,j^\prime =1}^s c_{ik}\otimes w_{kj}^{(i^\prime ,j^\prime )} \otimes d_{jj^\prime }
  + \sum\limits_{l,j^\prime ,j=1}^s w_{il}^{(i^\prime ,j^\prime )} \otimes d_{lj}\otimes d_{jj^\prime }  \\
 &
=& \sum\limits_{m,k=1}^r c_{im}\otimes c_{mk}\otimes x_k^{(i^\prime )}  \\
  & & + \sum\limits_{j^\prime ,j=1}^s  \left[
       \sum\limits_{k=1}^r c_{ik}\otimes w_{kj}^{(i^\prime ,j^\prime )}
       + \sum\limits_{l=1}^s w_{il}^{(i^\prime ,j^\prime )} \otimes d_{lj}  \right]  \otimes d_{jj^\prime } .
\end{eqnarray*}
We obtain
$$
\sum\limits_{j,j^\prime =1}^s \Delta(w_{ij}^{(i^\prime ,j^\prime )})\otimes d_{jj^\prime }
= \sum\limits_{j^\prime ,j=1}^s  \left[
       \sum\limits_{k=1}^r c_{ik}\otimes w_{kj}^{(i^\prime ,j^\prime )}
       + \sum\limits_{l=1}^s w_{il}^{(i^\prime ,j^\prime )} \otimes d_{lj}  \right]  \otimes d_{jj^\prime } .
$$
 So
$$
\Delta(w_{ij}^{(i^\prime ,j^\prime )})
= \sum\limits_{k=1}^r c_{ik}\otimes w_{kj}^{(i^\prime ,j^\prime )}
       + \sum\limits_{l=1}^s w_{il}^{(i^\prime ,j^\prime )} \otimes d_{lj} .
$$ for all $ 1\leq i^\prime ,i\leq r,1\leq j,j^\prime \leq s$.

Therefore, $\mathcal{W}^{(i^\prime ,j^\prime )}= \left( w_{ij}^{(i^\prime ,j^\prime )} \right)_{r\times s}$ is $\left(\mathcal{C},\mathcal{D}\right)$-primitive for each $1\leq i^\prime \leq r,1\leq j^\prime \leq s$. By equalities \eqref{E1} and \eqref{E2}, we have
$$
w=\sum\limits_{i=1}^r \sum\limits_{j=1}^s w_{ij}^{(i,j)}.
$$

\item
If $C=D$ and $\mathcal{C}=\mathcal{D}$, by Lemma~\ref{3}(2), there exists an element $c\in C$ such that
$$
\Delta\left( w-c \right) \in C\otimes({}^C H_1{}^C)^+ +({}^C H_1{}^C)^+ \otimes C .
$$
Then the same proof of (1) can be applied to the element $w-c$, and we are done.
\end{enumerate}
\end{proof}

\section{The Exponent of a Non-Cosemisimple Hopf Algebra with Chevalley Property in Characteristic 0}

We prove in this section that the exponent of a non-cosemisimple Hopf algebra in characteristic 0 is $\infty$ if its coradical is a finite-dimensional sub\-Hopf\-algebra.

Recall that the exponent for a Hopf algebra $H$ is the least positive integer $n$ such that the $n$th Hopf power $[n]$ is trivial. It can be defined on bialgebras, and it differs from the definition in P. Etingof and S. Gelaki \cite{E-G 1999}, which uses the antipode $S$.

\begin{definition}[\protect{\cite[Section 0]{Kashina 2000}} and \protect{\cite[Section 1]{L-M-S 2006}}]
Let $(H,m,u,\Delta,\varepsilon)$ be a ${\Bbbk}$-bialgebra,
\begin{enumerate}
\item For a positive integer $n$, denote following maps $\Delta_n:H\rightarrow H^{\otimes n},h\mapsto \sum\limits_h h_{(1)}\otimes h_{(2)}\otimes \cdots \otimes h_{(n)}$ and $m_n:H^{\otimes n}\rightarrow H,h_1\otimes h_2\otimes \cdots \otimes h_n\mapsto h_1h_2\cdots h_n$, where $h_1,h_2,\cdots,h_n\in H$. And the $n$th Hopf power map on $H$ is defined to be $[n]=m_n\circ \Delta_n$.
\item The $n$th Hopf power of $h\in H$ is said to be trivial if $h^{[n]}=\varepsilon(h)1$. The Hopf order of $h$ is the least positive integer $n$ such that the $n$th Hopf power of $h$ is trivial.
\item The exponent of $H$ is the least positive integer $n$ such that the $n$th Hopf power of every element in $H$ is trivial, and $\infty$ if such an $n$ does not exist. Denote the exponent of $H$ by ${\rm exp}(H)$.
\end{enumerate}
\end{definition}

Y. Kashina \cite{Kashina 1999} checked that $(\dim_{\Bbbk} H)$th Hopf power $[\dim_{\Bbbk} H]$ is trivial on $H$ if and only if it is trivial on a Hopf algebra $H\otimes_{\Bbbk} K$ over any extension field $K$ over ${\Bbbk}$. In fact the exponent is invariant under field extensions of bialgebras, and P. Etingof and S. Gelaki also listed this property (for exponent with $S^{-2}$) in \cite[Proposition 2.2(8)]{E-G 1999}.

\begin{proposition}[cf. \protect{\cite[Lemma 3]{Kashina 1999}}]\label{5}
Let $H$ be a ${\Bbbk}$-bialgebra. Let $K\supseteq{\Bbbk}$ be a field extension. Then ${\rm exp}(H\otimes_{\Bbbk} K)={\rm exp}(H)$.
\end{proposition}

\begin{proof}
For all $h\in H,\alpha\in K$ and for any positive integer $n$, $(h\otimes\alpha)^{[n]}=h^{[n]}\otimes\alpha$ and $u\circ\varepsilon(h\otimes\alpha)=\varepsilon(h)1_H\otimes\alpha$ in $H\otimes_{\Bbbk} K$. Thus, the $n$th Hopf power on $H\otimes_{\Bbbk} K$ is trivial if and only if the $n$th Hopf power on $H$ is trivial.
\end{proof}

Consider the $n$th Hopf power of a multiplicative matrix $\mathcal{G}=(g_{ij})$ over $H$. We need to calculate $\mathcal{G}^{[n]}=(g_{ij}{}^{[n]})=m_n\circ\Delta_n(\mathcal{G})$. First, $\Delta_n(\mathcal{G})=\mathcal{G}\widetilde{\otimes}\mathcal{G}\widetilde{\otimes}\cdots\widetilde{\otimes}\mathcal{G}$ ($n$ $\mathcal{G}$s in total). Then $m_n\circ\Delta_n(\mathcal{G})=\mathcal{G}^n$ by Proposition \ref{6}(2). So the $n$th Hopf power of a multiplicative matrix equals to the $n$th (multiplication) power of it.

The following proposition is an analogues of a result of Y. Kashina, Y. Sommerh\"{a}user and Y. Zhu~\cite[Corollary 3]{K-S-Z 2006}.

\begin{proposition}[cf. \protect{\cite[Corollary 3]{K-S-Z 2006}}]\label{19}
Let $(H,m,u,\Delta,\varepsilon)$ be a ${\Bbbk}$-bialgebra. Let $\mathcal{G}$ be an $r\times r$ multiplicative matrix over $H$. Then
\begin{enumerate}
\item $\mathcal{G}^{[n]}=m_n\circ\Delta_n(\mathcal{G})=\mathcal{G}^n$ for any positive integer $n$.
\item If $H$ is  a Hopf algebra with antipode $S$, then $S(\mathcal{G})\mathcal{G}=\mathcal{G}S(\mathcal{G})=I_r$.
\end{enumerate}
\end{proposition}

\begin{proof}
We prove (2). Since $\Delta(\mathcal{G})=\mathcal{G}\widetilde{\otimes}\mathcal{G}$, then $S(\mathcal{G})\mathcal{G}=m\circ(S\otimes id)\circ\Delta(\mathcal{G})=u\circ\varepsilon(\mathcal{G})=u(I_r)=I_r$.
\end{proof}

Now we consider a non-cosemisimple  Hopf algebra $H$ over ${\Bbbk}$.

\begin{proposition}\label{4}
Let $H$ be a non-cosemisimple Hopf algebra. Then $H_1{}^1\supsetneq {\Bbbk} 1$, i.e. $\left({}^C H_1{}^1\right)^+\neq 0$ for some $C\in\mathcal{S}$.
\end{proposition}

\begin{proof}
Suppose that  $H_1{}^1={\Bbbk} 1$. We will conclude a contradiction.

We first show that $H^1={\Bbbk} 1$. Obviously we have $H_0{}^1={\Bbbk} 1$ and $H_1{}^1={\Bbbk} 1$. We will show by induction that $H_n{}^1={\Bbbk} 1$ for any $n\in {\mathbb{ Z}}^+$.

Assume that we have proved  $H_0{}^1=H_1{}^1=\cdots=H_{n-1}{}^1={\Bbbk} 1$. Then
\begin{eqnarray*}
   \Delta(H_n{}^1)   &  \subseteq& H_0 \otimes H_n{}^1 + \sum\limits_{i=1}^n H_i \otimes H_{n-i}{}^1  \\
                     &  =&         H_0 \otimes H_n{}^1 + \sum\limits_{i=1}^n H_i \otimes {\Bbbk} 1  \\
                     &  \subseteq& H_0 \otimes H_n{}^1 + H_n \otimes {\Bbbk} 1  \\
                     &  \subseteq& H_0 \otimes H + H \otimes H_0.
\end{eqnarray*}
Therefore by coradical definition, $H_n{}^1 \subseteq H_1$, and $H_n{}^1 = (H_n{}^1)^1 \subseteq H_1{}^1$ for all $n\in {\mathbb{ Z}}^+$. As a consequence, $H_n{}^1 = H_1{}^1 = {\Bbbk} 1$.

Now $H^1={\Bbbk} 1$ holds. By Proposition~\ref{2} (4),
$$
H=H^1\oplus\left(\bigoplus\limits_{D\in\mathcal{S}\setminus\{{\Bbbk}1\}}H^D\right)= {\Bbbk}1\oplus \left(\bigoplus\limits_{D\in\mathcal{S}\setminus \{{\Bbbk} 1\} } {H^D}\right)
$$
Consider $e_{{\Bbbk} 1}\in H^\ast$. It is clear that
$$
\left\langle e_{{\Bbbk} 1}, \bigoplus\limits_{D\in\mathcal{S}\setminus \{{\Bbbk} 1\} } H^D \right\rangle
=\sum\limits_{D\in\mathcal{S}\setminus \{{\Bbbk} 1\} } \left\langle e_{{\Bbbk} 1}, H^D \right\rangle
=\sum\limits_{D\in\mathcal{S}\setminus \{{\Bbbk} 1\} } \left\langle e_{{\Bbbk} 1} e_D , H \right\rangle
=0 .
$$
Note that $\bigoplus\limits_{D\in\mathcal{S}\setminus \{{\Bbbk} 1\} }{H^D}$ is a left coideal of $H$. It follows that $e_{{\Bbbk} 1}$ is a left integral in $H^\ast$. However, $\langle e_{{\Bbbk} 1},1 \rangle =1$ implies the cosemisimplity of $H$ by dual Maschke Theorem \cite[Theorem 2.4.6]{Montgomery 1993}, which contracts with the assumption that $H$ is non-cosemisimple. Consequently $H_1^{~1}\supsetneq {\Bbbk} 1$.

By a result of Radford \cite[Formula (1.5c)]{Radford 1978}, we have ${}^1 H_1{}^1={\Bbbk} 1\oplus \left({}^1 H_1{}^1\right)^+$. On the other hand, $\left( {}^C H_1{}^1 \right)^+ ={}^C H_1{}^1$ for each $C\in \mathcal{S}\setminus\{{\Bbbk} 1\}$ holds by Proposition \ref{2} (3). Therefore
\begin{eqnarray*}
H_1{}^1 & =&\bigoplus\limits_{C\in\mathcal{S}} {}^C H_1{}^1
={}^1 H_1{}^1 \oplus\left(\bigoplus\limits_{C\in\mathcal{S}\setminus\{{\Bbbk} 1\}} {}^C H_1{}^1 \right)  \\
 &
=&{\Bbbk} 1\oplus\left({}^1 H_1{}^1 \right)^+ \oplus\left[\bigoplus\limits_{C\in\mathcal{S}\setminus\{{\Bbbk} 1\}} \left({}^C H_1{}^1 \right)^+\right]  \\
 &
=&{\Bbbk} 1\oplus\left[\bigoplus\limits_{C\in\mathcal{S}} \left({}^C H_1{}^1 \right)^+\right]
\supsetneq {\Bbbk} 1
\end{eqnarray*}
by Proposition \ref{2} (4). Thus we have $\left({}^C H_1{}^1 \right)^+\neq 0$ for some $C\in\mathcal{S}$.
\end{proof}

Note that if $H$ is a ${\Bbbk}$-Hopf algebra, then $H$ is cosemisimple if and only if $H\otimes_{\Bbbk} K$ is cosemisimple for any field extension $K/{\Bbbk}$ (R.G. Larson \cite[Lemma 1.3]{Larson 1971}). One observe that the exponent of $H$ is also invariant under field extensions by Proposition \ref{5}.

\begin{theorem}\label{9}
Let $H$ be a non-cosemisimple Hopf algebra over a field ${\Bbbk}$ of characteristic $0$. If the coradical $H_0$ is a finite-dimensional sub\-Hopf\-algebra (the Chevalley property), then there exists an element in $H_1{}^1$ having infinite Hopf order. Thus ${\rm exp}(H)=\infty$.
\end{theorem}

\begin{proof}
Without loss of generality, we assume that ${\Bbbk}$ is algebraically closed, and then every simple subcoalgebra of $H$ has a basic multiplicative matrix. The non-cosemisimplity of $H$ implies that there exists a simple subcoalgebra $C\in\mathcal{S}$ such that $\left({}^C H_1{}^1\right)^+\neq 0$ by Proposition \ref{4}.

If $C={\Bbbk} 1$, then we choose a non-zero element $w\in\left({}^1 H_1{}^1\right)^+$. It's straightforward to verify that $\Delta(w)=1\otimes w+w\otimes 1$ (in fact, $\left({}^1 H_1{}^1\right)^+$ consists of all primitive elements). Thus $w^{[n]}=nw\neq 0$ for each positive integer $n$.

 If $C\neq{\Bbbk} 1$, we need to find an element with infinite Hopf order. Choose an $r\times r$ basic multiplicative matrix $\mathcal{C}=(c_{ij})$ of $C$, and choose a non-zero element $w\in \left({^C}H_1^{~1}\right)^+$. By Theorem \ref{7} (1), $w$ is the sum of $r$ elements, each of which is an entry of a $(\mathcal{C},1)$-primitive matrix. So one of those $(\mathcal{C},1)$-primitive matrices is not zero. We denoted this matrix by $\mathcal{W}\in\mathcal{M}_{r\times 1}({}^C H_1{}^1)$. Apparently,
$\left(
\begin{array}{cc}
  \mathcal{C} & \mathcal{W} \\
  0 & 1
\end{array}
\right)$
is multiplicative. So its $n$th Hopf power is
$$
\left(
\begin{array}{cc}
  \mathcal{C} & \mathcal{W} \\
  0 & 1
\end{array}
\right)^{[n]}
=
\left(
\begin{array}{cc}
  \mathcal{C} & \mathcal{W} \\
  0 & 1
\end{array}
\right)^n
=
\left(
\begin{array}{cc}
  \mathcal{C}^n & (I_r+\mathcal{C}+\mathcal{C}^2+\cdots+\mathcal{C}^{n-1})\mathcal{W} \\
  0 & 1
\end{array}
\right).
$$
by Proposition \ref{19} (1).

It remains to show that $\mathcal{W}^{[n]}=(I_r+\mathcal{C}+\mathcal{C}^2+\cdots+\mathcal{C}^{n-1})\mathcal{W}\neq 0=\varepsilon(\mathcal{W})$ for each positive integer $n$. Consider the finite-dimensional sub\-Hopf\-algebra $H_0$. It is semisimple since it is cosemisimple and ${\rm char}~{\Bbbk}=0$. Hence $H_0$ is involutory by \cite[Theorem 3]{L-R 1998} and unimodular. Using Proposition \ref{8} (3), we have a non-zero integral $\Lambda_0=1+\sum\limits_{D\in\mathcal{S}\setminus\{{\Bbbk} 1\}} r_D t_D$ of $H_0$, where $r_D$ is a positive integer and $t_D\in D$ for each $D\in\mathcal{S}\setminus\{{\Bbbk} 1\}$.

Since $\Lambda_0$ is a right integral of $H_0$, then
\begin{eqnarray*}
\Lambda_0 \mathcal{W}^{[n]}  &
=& \Lambda_0 (I_r+\mathcal{C}+\mathcal{C}^2+\cdots+\mathcal{C}^{n-1})\mathcal{W}  \\
&  =& \left( \sum\limits_{i=0}^{n-1} \Lambda_0 \mathcal{C}^i \right)\mathcal{W}  \\
&  = &\left( \sum\limits_{i=0}^{n-1} \Lambda_0 I_r\right)\mathcal{W}  \\
&  =& n\Lambda_0 \mathcal{W}.
\end{eqnarray*}
So it suffices to show that $\Lambda_0\mathcal{W}\neq 0$.

Let $x\in \left({}^C H_1{}^1 \right)^+$ be a non-zero entry of $\mathcal{W}$. We will show that $\Lambda_0 x\neq 0$. First by $\Delta(x)\in {}^C H_1{}^1 \otimes 1 + C \otimes {}^C H_1{}^1$,
\begin{eqnarray*}
\Delta \left( t_D x \right)   &
\in& D ({}^C H_1{}^1) \otimes D + DC \otimes D ({}^C H_1{}^1)   \\
 &
\subseteq& H_1 \otimes D + H_0 \otimes H_1 ,
\end{eqnarray*}
$\forall D\in\mathcal{S}\setminus\{{\Bbbk} 1\}$,
where the last inclusion holds since $\{H_n\}_{n=0}^\infty$ is also an algebra filtration, which is followed by the Chevalley property of $H$ (Montgomery \cite[Lemma 5.2.8]{Montgomery 1993}).
Then for any
$D\in\mathcal{S}\setminus\{{\Bbbk} 1\},~\left( t_D x \right)^1 \subseteq H_1 \langle e_{{\Bbbk}1},D\rangle + H_0\langle e_{{\Bbbk}1},H_1\rangle \subseteq H_0$.
Thus
\begin{eqnarray*}
(\Lambda_0 x)^1   &
=&[ ( 1_H + \sum\limits_{D\in \mathcal{S}\setminus\{{\Bbbk} 1\}} r_D t_D ) x ]^1 \\
 &
= &x^1 + \sum\limits_{D\in \mathcal{S}\setminus\{{\Bbbk} 1\}} r_D ( t_D x )^1   \\
 &
\subseteq& x + H_0.
\end{eqnarray*}
But $0 \neq x \in {}^C H_1{}^1$ implies $0 \notin x + H_0$ because $H_0\oplus {}^C H_1{}^1$ is direct by Proposition \ref{2}(3). So $\Lambda_0 x \neq 0$, and $\Lambda_0 \mathcal{W}^{[n]}=n\Lambda_0\mathcal{W}\neq 0$ holds for each positive integer $n$. The proof is completed.
\end{proof}

\section{The Exponent of a Finite-Dimensional (Non-Co\-semi\-simple) Pointed Hopf Algebra in Positive Characteristic}

We prove in this section that the exponent of a finite-dimensional pointed Hopf algebra $H$ in characteristic $p>0$ is finite. The proof bases on decomposing an arbitrary element as a sum of elements, each of which is an entry of a multiplicative matrix.

First, it should be pointed out that the definition of multiplicative matrices over a coalgebra $H$ can be extended to block case. A block matrix
$$
\mathcal{G}=
\left(
  \begin{array}{cccc}
    \mathcal{G}_{11} & \mathcal{G}_{12} & \cdots & \mathcal{G}_{1r}     \\
    \mathcal{G}_{21} & \mathcal{G}_{22} & \cdots & \mathcal{G}_{2r}     \\
    \vdots & \vdots & \ddots & \vdots     \\
    \mathcal{G}_{r1} & \mathcal{G}_{r2} & \cdots & \mathcal{G}_{rr}
  \end{array}
\right)
$$
over $H$ is called a \emph{block multiplicative matrix} if $\mathcal{G}_{ii}$ is a square block for each $1\leq i\leq r$ and
$\Delta(\mathcal{G}_{ij})=\sum\limits_{i=1}^r \mathcal{G}_{ik}\widetilde{\otimes}\mathcal{G}_{kj}$, $\varepsilon(\mathcal{G}_{ij})=\left\{\begin{array}{ll} \text{the identity matrix}\ I,& \text{if}\  i=j\\
\text{the zero matrix}\ 0, & \text{if}\  i\neq j
\end{array}\right.$ for each $1\leq i,j\leq r$. It is straightforward to verify that a block matrix is a block multiplicative matrix if and only if it is a usual multiplicative matrix. The following proposition shows that any upper triangular block multiplicative matrix over a bialgebra has finite order or Hopf order.

\begin{proposition}\label{10}
Let $H$ be a ${\Bbbk}$-bialgebra, where ${\rm char}~{\Bbbk}=p>0$. Assume
$$
\mathcal{Z}=
\left(
  \begin{array}{cccccc}
    \mathcal{G}_0 & \mathcal{W}_{0\,1} & \mathcal{W}_{0\,2} & \cdots & \mathcal{W}_{0\,n-1} & \mathcal{W}_{0\,n}       \\
    0   & \mathcal{G}_1    & \mathcal{W}_{1\,2} & \cdots & \mathcal{W}_{1\,n-1} & \mathcal{W}_{1\,n}    \\
    0   & 0      & \mathcal{G}_2    & \cdots & \mathcal{W}_{2\,n-1} & \mathcal{W}_{2\,n}    \\
    \vdots & \vdots & \vdots &\ddots & \vdots & \vdots     \\
    0   & 0      & 0      & \cdots & \mathcal{G}_{n-1}   & \mathcal{W}_{n-1\,n} \\
    0   & 0      & 0      & \cdots & 0         & \mathcal{G}_n
  \end{array}
\right)
$$
is a multiplicative block matrix over $H$. If there exists a positive integer $d$ such that $\mathcal{G}_0^{~d}=I,\mathcal{G}_1^{~d}=I,\ldots,$ and $\mathcal{G}_n^{~d}=I$, then $\mathcal{Z}^{\left[ d p^{ \lfloor \log_p n \rfloor +1 } \right]}=\mathcal{Z}^{d p^{ \lfloor \log_p n \rfloor +1 }}=I$, where $\lfloor \cdot \rfloor$ denotes the floor function.
\end{proposition}

\begin{proof}
It is clear that
$$
\mathcal{Z}^d=
\left(
  \begin{array}{cccccc}
    I   & \mathcal{W}^{\prime}_{0\,1} & \mathcal{W}^{\prime}_{0\,2} & \cdots & \mathcal{W}^{\prime}_{0\,n-1} & \mathcal{W}^{\prime}_{0\,n}    \\
    0   & I                           & \mathcal{W}^{\prime}_{1\,2} & \cdots & \mathcal{W}^{\prime}_{1\,n-1} & \mathcal{W}^{\prime}_{1\,n}    \\
    0   & 0                           & I                           & \cdots & \mathcal{W}^{\prime}_{2\,n-1} & \mathcal{W}^{\prime}_{2\,n}    \\
    \vdots & \vdots & \vdots &\ddots & \vdots & \vdots     \\
    0   & 0                           & 0                           & \cdots & I                             & \mathcal{W}^{\prime}_{n-1\,n}  \\
    0   & 0                           & 0                           & \cdots & 0                             & I
  \end{array}
\right),
$$
where $\mathcal{W}^{\prime}_{i\,j}$ have the same sizes as $\mathcal{W}_{i\,j}$, for all $0\leq i\leq j\leq n$. Set
$$
\mathcal{W}=
\left(
  \begin{array}{cccccc}
    0   & \mathcal{W}^{\prime}_{0\,1} & \mathcal{W}^{\prime}_{0\,2} & \cdots & \mathcal{W}^{\prime}_{0\,n-1} & \mathcal{W}^{\prime}_{0\,n}    \\
    0   & 0                           & \mathcal{W}^{\prime}_{1\,2} & \cdots & \mathcal{W}^{\prime}_{1\,n-1} & \mathcal{W}^{\prime}_{1\,n}    \\
    0   & 0                           & 0                           & \cdots & \mathcal{W}^{\prime}_{2\,n-1} & \mathcal{W}^{\prime}_{2\,n}    \\
    \vdots & \vdots & \vdots &\ddots & \vdots & \vdots     \\
    0   & 0                           & 0                           & \cdots & 0                             & \mathcal{W}^{\prime}_{n-1\,n}  \\
    0   & 0                           & 0                           & \cdots & 0                             & 0
  \end{array}
\right),
$$
and it is easy to show that $\mathcal{W}^{n+1}=0$. Thus
$$
\mathcal{Z}^{d p^{ \lfloor \log_p n \rfloor +1 }}=(I+\mathcal{W})^{p^{ \lfloor \log_p n \rfloor +1 }}=I+\mathcal{W}^{p^{ \lfloor \log_p n \rfloor +1 }}=I,
$$
since ${\rm char}~{\Bbbk}=p$ and $p^{ \lfloor \log_p n \rfloor +1 }>n$. Finally
$\mathcal{Z}^{\left[ d p^{ \lfloor \log_p n \rfloor +1 } \right]}=\mathcal{Z}^{d p^{ \lfloor \log_p n \rfloor +1 }}$
 follows from that $\mathcal{Z}$ is multiplicative by Proposition \ref{19} (1).
\end{proof}

However one cannot expect, like  $\mathcal{W}_{0\,n}$ in the Proposition \ref{10}, that any element in $H_n$ can be expressed as an entry of the top right corner of one  $n\times n$ upper triangular multiplicative (block) matrix over $H$ or sums of the top right corner of finite number of such matrices. As we will see, we have to embedded $H$ into a larger coalgebra. This technique relies on the notion of free bialgebras on coalgebras (for knowledge of free bialgebras, see Radford \cite[Definition 5.3.2]{Radford 2012}).

A \emph{free bialgebra on the coalgebra} $H$ over ${\Bbbk}$ is a pair $(\iota_H,T(H))$ which satisfies:
\begin{enumerate}
\item $T(H)$ is a ${\Bbbk}$-bialgebra and $\iota_H:H\hookrightarrow T(H)$ is a coalgebra map;
\item If $A$ is a ${\Bbbk}$-bialgebra and $f:H\rightarrow A$ is a coalgebra map, then there exists a unique bialgebra map $F:T(H)\rightarrow A$ such that $F\circ\iota_H=f$.
\end{enumerate}

In fact, a free bialgebra on $(H,\Delta,\varepsilon)$ can be chosen as $T(H)=\bigoplus\limits_{n=0}^\infty H^{\otimes n}$, which is the tensor algebra on the vector space $H$. Comultiplication $\Delta_{T(H)}$ and counit $\varepsilon_{T(H)}$ on $T(H)$ are algebra maps determined by $\varepsilon_{T(H)}\circ\iota=\varepsilon$ and $\Delta_{T(H)}\circ\iota_H=(\iota_H\otimes\iota_H)\circ\varepsilon$ respectively.

\begin{lemma}\label{17}
Let $H$ be a bialgebra over ${\Bbbk}$. Let $T(H)=\bigoplus\limits_{n=0}^\infty H^{\otimes n}$ be the free bialgebra on $H$ as a coalgebra. Then
\begin{enumerate}
\item
There exists a unique bialgebra map $\pi_H:T(H)\twoheadrightarrow H$ making the following diagram commute:
$$
\xymatrix{
  H \ar[dr]_-{id_H} \ar@{^{(}->}[r]^-{\iota_H}
                & T(H) \ar[d]^-{\pi_H}  \\
                & H            }.
$$

In fact, $\pi_H(h_1\otimes h_2\otimes\cdots\otimes h_n)= h_1h_2\cdots h_n$ and $\pi_H(1)=1_H$, where $h_1,h_2,\cdots,h_n\in H$. Therefore $H\cong T(H)/{\rm Ker}~\pi_H$, and $\varphi:h\mapsto h+{\rm Ker~\pi_H}$ is a bialgebra isomorphism.
\item
Let a vector-space direct-sum $\overline{H}= H\oplus V$ be a ${\Bbbk}$-coalgebra, such that $H\subseteq\overline{H}$ is a subcoalgebra. Let $T(\overline{H})=\bigoplus\limits_{n=0}^\infty \overline{H}^{\otimes n}$ be the free bialgebra on $\overline{H}$. Then
    \begin{enumerate}[(i)]
    \item $T(\overline{H})({\rm Ker}~\pi_H)T(\overline{H})$ is a biideal of $T(\overline{H})$;
    \item $[T(\overline{H})({\rm Ker}~\pi_H)T(\overline{H})]\cap T(H)={\rm Ker}~\pi_H$.

    So the bialgebra inclusion $T(H)\hookrightarrow T(\overline{H})$ leads to the bialgebra monomorphism $\phi: T(H)/{\rm Ker}~\pi_H \hookrightarrow T(\overline{H})/[T(\overline{H})({\rm Ker}~\pi_H)T(\overline{H})]$.
    \item The following diagram commutes:
    $$
    {
    \xymatrix{
    H  \ar@{^{(}->}[d]_-{\iota_{\overline{H}}|_H}  \ar[r]^-\varphi  &  T(H)/({\rm Ker}~\pi_H)     \ar@{^{(}->}[d]^-\phi  \\
     T(\overline{H})  \ar[r]_-\theta   &  T(\overline{H})/\left[ T(\overline{H})({\rm Ker}~\pi_H)T(\overline{H})\right]             }
    }
    $$
    where $\varphi$ is the bialgebra isomorphism, $\varphi$ is the bialgebra isomorphism, $\phi$ is the bialgebra monomorphism, $\theta$ is the quotient bialgebra map, and $\iota_{\overline{H}}|_H$ is just a coalgebra inclusion.
    \end{enumerate}
\end{enumerate}
\end{lemma}

\begin{proof}\hspace{1ex}\\
\begin{enumerate}
\item Regard $id_H$ as a coalgebra epimorphism from coalgebra $H$ onto bialgebra $H$. Then by the universal property of free bialgebras, there exists a unique surjective bialgebra map $\pi_H:T(H)\rightarrow H$ such that $\pi_H\circ\iota_H=id_H$.
\item     \begin{enumerate}[(i)]
    \item First we show that $T(\overline{H})({\rm Ker}~\pi_H)T(\overline{H})\subseteq T(\overline{H})$ is a coideal and thus a biideal. As ${\rm Ker}~\pi_H$ is a coideal of $T(H)$, we have $\varepsilon_{T(H)}({\rm Ker}~\pi_H)=0$ and then $\varepsilon_{T(\overline{H})}({\rm Ker}~\pi_H)=0$. Since $\varepsilon_{T(\overline{H})}$ is an algebra map, which follows that
        $$\varepsilon_{T(\overline{H})}(T(\overline{H})({\rm Ker}~\pi_H)T(\overline{H}))=0.$$

        On the other hand, it is clear that $T(\overline{H})({\rm Ker}~\pi_H)T(\overline{H})$ is an ideal of $T(\overline{H})$, so the quotient map $\theta:T(\overline{H})\twoheadrightarrow T(\overline{H})/[T(\overline{H})({\rm Ker}~\pi_H)T(\overline{H})]$ is an algebra epimorphism. Since $\theta$ and $\Delta_{T(\overline{H})}$ are both algebra maps, then
        \begin{eqnarray*}
        && (\theta\otimes\theta)\circ\Delta_{T(\overline{H})}(T(\overline{H})({\rm Ker}~\pi_H)T(\overline{H})) \\
        &\subseteq& [T(\overline{H})\otimes T(\overline{H})][(\theta\otimes\theta)\circ\Delta_{T(\overline{H})}({\rm Ker}~\pi_H)][T(\overline{H})\otimes T(\overline{H})] \\
        &=& [T(\overline{H})\otimes T(\overline{H})][(\theta\otimes\theta)\circ\Delta_{T(H)}({\rm Ker}~\pi_H)][T(\overline{H})\otimes T(\overline{H})] \\
        &=&0,
        \end{eqnarray*}
        where the last equation holds also because ${\rm Ker}~\pi_H$ is a coideal of $T(H)$ which implies
        $$\Delta_{T(\overline{H})}({\rm Ker}~\pi_H)=\Delta_{T(H)}({\rm Ker}~\pi_H)\subseteq {\rm Ker}~\pi_H\otimes T(H)+T(H)\otimes{\rm Ker}~\pi_H,$$ and $\theta({\rm Ker}~\pi_H)=0$. Hence $T(\overline{H})({\rm Ker}~\pi_H)T(\overline{H})$ is a coideal of $T(\overline{H})$.

    \item
        Let $\eta$ denote the ${\Bbbk}$-linear projection from $\overline{H}= H\oplus V$ onto $H$. Then by the universal property, there exists a unique algebra epimorphism $\overline{\eta}:T(\overline{H})\twoheadrightarrow T(H)$ making the following diagram commute ($\overline{\eta}$ will be a bialgebra map if $V$ is in addition a coideal of $\overline{H}$):
        $$
        {
        \xymatrix{
        \overline{H}  \ar[d]_-{\eta}  \ar@{^{(}->}[r]^-{\iota_{\overline{H}}}  &  T(\overline{H})  \ar[d]^-{\overline{\eta}}  \\
        H  \ar@{^{(}->}[r]_-{\iota_H}  &  T(H)             }
        },
        $$
        In fact $\overline{\eta}(a_1\otimes a_2\otimes\cdots\otimes a_n)=\eta(a_1)\otimes \eta(a_2)\otimes\cdots\otimes \eta(a_n)$, where $a_1,a_2,\cdots,a_n\in\overline{H}$. Moreover,  $\overline{\eta}|_{T(H)}=id_{T(H)}$ if we regard $T(H)\subseteq T(\overline{H})$ which is a bialgebra inclusion.

        In order to show that $[T(\overline{H})({\rm Ker}~\pi_H)T(\overline{H})]\cap T(H)={\rm Ker}~\pi_H$, we only need to prove ${\rm Ker}~\pi_H\subseteq [T(\overline{H})({\rm Ker}~\pi_H)T(\overline{H})]\cap T(H)$. Let $a\triangleq \sum\limits_{i=1}^n a_i\otimes h_i\otimes b_i$ be an arbitrary element in $T(\overline{H})({\rm Ker}~\pi_H)T(\overline{H})$, where $a_i,b_i\in T(\overline{H})$ and $h_i\in{\rm Ker}~\pi_H$ for  $1\leq i\leq n$. Then
        $$
        \overline{\eta}(a)
        = \overline{\eta}\left( \sum\limits_{i=1}^n a_i\otimes h_i\otimes b_i\right)
        =\sum\limits_{i=1}^n \overline{\eta}(a_i)\otimes \overline{\eta}(h_i)\otimes \overline{\eta}(b_i)
        =\sum\limits_{i=1}^n \overline{\eta}(a_i)\otimes h_i\otimes \overline{\eta}(b_i) .
        $$
        So if $a\in T(H)$, then
        $\pi_H(a)=\sum\limits_{i=1}^n [\pi_H\circ\overline{\eta}(a_i)]\otimes\pi_H(h_i)\otimes[\pi_H\circ\overline{\eta}(b_i)]=0$. Thus $[T(\overline{H})({\rm Ker}~\pi_H)T(\overline{H})]\cap T(H)\subseteq {\rm Ker}~\pi_H$, and the equation then follows.
    \item It is straight forward to verify that the diagram commutes.
    \end{enumerate}
\end{enumerate}
\end{proof}

\begin{lemma}\label{18}
Let $H$ be a coalgebra. Assume for each $i\in\Gamma$, $H_{\langle i\rangle}=H\oplus V_{\langle i\rangle}$ is a coalgebra such that the inclusion $H\subseteq H_{\langle i\rangle}$ is a coalgebra map.
Then we can make $\overline{H}\triangleq H\oplus\left(\bigoplus\limits_{i\in\Gamma} V_{\langle i\rangle}\right)$  a coalgebra such that the inclusion $H_{\langle i\rangle}\subseteq \overline{H}$ is a coalgebra map, for each $i\in\Gamma$.
\end{lemma}

\begin{proof}
Since $(H,\Delta,\varepsilon)$, $(H_{\langle i\rangle},\Delta_i,\varepsilon_i)$ ($i\in\Gamma$) are all coalgebras, we define linear maps $\overline{\Delta}$ and $\overline{\varepsilon}$ on $\overline{H}$ to be $\Delta_i$ and $\varepsilon_i$ respectively when they are restricted on $V_{\langle i\rangle}$, and $\overline{\Delta}|_H=\Delta,\overline{\varepsilon}|_H=\varepsilon$. Then $(\overline{H},\overline{\Delta},\overline{\varepsilon})$ is obviously a coalgebra. Clearly, $H_{\langle i\rangle}$ is a subcoalgebra of $\overline{H}$.
\end{proof}

From now on,  $H$ will be a finite-dimensional pointed ${\Bbbk}$-coalgebra, and we let $G(H)$ to denote all group-like elements of $H$. Take a family of coradical orthonormal idempotents  $\{e_g\}_{g\in G(H)}$ of $H^\ast$.

\begin{proposition}\label{11}
Let $H$ be a finite-dimensional pointed ${\Bbbk}$-coalgebra. Then for any positive integer $n$, if $g,h\in G(H)$ and $z\in \left({}^g H_n{}^h\right)^+$, then
$$
\Delta(z)
= g\otimes z + \sum\limits_{i=1}^{n-1} \sum\limits_{j\in J_i} \sum\limits_{k\in G(H)} x_{i,j}^{(k)}\otimes y_{i,j}^{(k)}+ z\otimes h ,
$$
where  $J_i$ is a finite set for each $i$,  $x_{i,j}^{(k)}\in\left({}^g H_i{}^k\right)^+ ,\ y_{i,j}^{(k)}\in \left( {}^k H_{n-i}{}^h \right)^+$ for each $1\leq i\leq n-1$, $j\in J_i$ and $k\in G(H)$.
\end{proposition}

\begin{proof}
For  $z\in {}^g H_n{}^h$, $z={}^g z^h=\sum\limits_z \langle e_g,z_{(1)} \rangle z_{(2)} \langle e_h,z_{(3)} \rangle$. Since $\sum\limits_{k\in G(H)}e_k=\varepsilon$,
\begin{eqnarray*}
\Delta(z)  &
=& \sum\limits_z \langle e_g,z_{(1)} \rangle z_{(2)}\otimes z_{(3)} \langle e_h,z_{(4)} \rangle  \\
  &
= &\sum\limits_z \sum\limits_{k\in G(H)} \langle e_g,z_{(1)} \rangle z_{(2)} \langle e_k,z_{(3)} \rangle
\otimes \langle e_k,z_{(4)} \rangle z_{(5)} \langle e_h,z_{(6)} \rangle  \\
  &
= &\sum\limits_z \sum\limits_{k\in G(H)} {}^g z_{(1)}{}^k \otimes {}^k z_{(2)}{}^h.  \\
\end{eqnarray*}
It turns out that
$$
\Delta({}^g H_n{}^h)
\subseteq \sum\limits_{i=0}^n \sum\limits_{k\in G(H)} {}^g H_i{}^k \otimes{}^k H_{n-i}{}^h
=g \otimes H_n{}^h + \sum\limits_{i=1}^{n-1} \sum\limits_{k\in G(H)} {}^g H_i{}^k \otimes{}^k H_{n-i}{}^h + {}^g H_n \otimes h.
$$
Therefore we can write
$$
\Delta(z) = g\otimes z^\prime + \sum\limits_{i=1}^{n-1} \sum\limits_{j\in J_i} \sum\limits_{k\in G(H)} x_{i,j}^{(k)}\otimes y_{i,j}^{(k)}+ z^{\prime\prime}\otimes h,
$$
where $J_i$ is a finite set, and  $x_{i,j}^{(k)}\in{}^g H_i{}^k ,\ y_{i,j}^{(k)}\in {}^k H_{n-i}{}^h$ for all $1\leq i\leq n-1$, $j\in J_i$ and $k\in G(H)$.

On the other hand, it is straightforward to show that ${}^g H_n{}^g={\Bbbk} g\oplus\left({}^g H_n{}^g\right)^+$ for each $g\in G(H)$, and ${}^g H_n{}^h=\left({}^g H_n{}^h\right)^+$ if $g\neq h$. Denote ${^+}=id-g\,\delta_{g,h}\,\varepsilon:{}{}^g H_i{}^h\twoheadrightarrow\left( {}^g H_i{}^h\right)^+$ to be the ${\Bbbk}$-linear projection.

Then
\begin{eqnarray}
\Delta(z)
& =& g\otimes z^\prime + \sum\limits_{i=1}^{n-1}\sum\limits_{j\in J_i}\sum\limits_{k\in G(H)}
      x_{i,j}^{(k)} \otimes y_{i,j}^{(k)} + z^{\prime\prime}\otimes h \label{E3} \\
 &=& g\otimes z^\prime + \sum\limits_{i=1}^{n-1}\sum\limits_{j\in J_i}\sum\limits_{k\in G(H)}
      \left[ \left( x_{i,j}^{(k)} \right)^+ + \varepsilon(x_{i,j}^{(k)})g \right]
      \otimes
      \left[ \left( y_{i,j}^{(k)} \right)^+ + \varepsilon(y_{i,j}^{(k)})h \right]
     + z^{\prime\prime}\otimes h.  \nonumber
\end{eqnarray}

Now if we set
\begin{eqnarray*}
\widetilde{z}^\prime
&=& z^\prime + \sum\limits_{i=1}^{n-1}\sum\limits_{j\in J_i}\sum\limits_{k\in G(H)}
  \varepsilon(x_{i,j}^{(k)})
  \left[ \left( y_{i,j}^{(k)} \right)^+ + \varepsilon(y_{i,j}^{(k)})h \right]
  \in H_n ,  \\
\widetilde{z}^{\prime\prime}
&=& z^{\prime\prime} + \sum\limits_{i=1}^{n-1}\sum\limits_{j\in J_i}\sum\limits_{k\in G(H)}
  \left( x_{i,j}^{(k)} \right)^+ \varepsilon(y_{i,j}^{(k)}) \in H_n ~,
\end{eqnarray*}
then Equation \eqref{E3} becomes
\begin{equation}\label{E4}
\Delta(z)
 = g\otimes \widetilde{z}'
   + \sum\limits_{i=1}^{n-1}\sum\limits_{j\in J_i}\sum\limits_{k\in G(H)}
       \left( x_{i,j}^{(k)} \right)^+
         \otimes
       \left( y_{i,j}^{(k)} \right)^+
   + \widetilde{z}''\otimes h .
\end{equation}
Now we let $z\in\left({^g}H_n^{~h}\right)^+$. By applying $\varepsilon\otimes id$ and $id\otimes \varepsilon$ respectively, we obtain
$\widetilde{z}^\prime=z-\varepsilon(\widetilde{z}^{\prime\prime})h$ and $\widetilde{z}^{\prime\prime}=z-\varepsilon(\widetilde{z}^\prime)g$. So
$$
\Delta(z)
 = g\otimes z
   + \sum\limits_{i=1}^{n-1}\sum\limits_{j\in J_i}\sum\limits_{k\in G(H)}
       \left( x_{i,j}^{(k)} \right)^+
         \otimes
       \left( y_{i,j}^{(k)} \right)^+
   + z\otimes h
   - [\varepsilon(\widetilde{z}^\prime)+\varepsilon(\widetilde{z}^{\prime\prime})] g\otimes h .
$$
Note that $\varepsilon(\widetilde{z}^\prime)+\varepsilon(\widetilde{z}^{\prime\prime})=(\varepsilon\otimes\varepsilon)\circ\Delta(z)=0$. This finishes the proof.
\end{proof}

As usual, for $g,h\in G(H)$, we let $P_{g,h}(H)=\{w\in H \mid \Delta(w)=g\otimes w+w\otimes h\}$ denote the set of all $g,h$-primitive elements in $H$.

\begin{definition}
Let $H$ be a coalgebra and $g,h\in G(H)$. Define a $(g,h)$-upper triangular multiplicative matrix (over $H$) to be an upper triangular multiplicative matrix in which the top left entry is $g$ and the bottom right entry is $h$.
\end{definition}

Now for a given $z\in\left({}^g H_n{}^h\right)^+$, we will embed $(H,\Delta_H,\varepsilon_H)$ into a larger coalgebra $H_{\langle n\rangle}$, in which $z$ can be written as a sum of finite number of elements,  with each of them being the $(1,n+1)$-entry of a $(g,h)$-upper triangular matrix over $H_{\langle n\rangle}$. Since $\left({}^g H_0{}^h\right)^+=0$ for all $g,h\in G(H)$, only consider the case when $n\geq 1$. We begin with the case $n=1,2$.

\begin{definition}\label{12}
Assume that $g,h\in G(H)$.
\begin{enumerate}
\item For $z\in\left({}^g H_1{}^h\right)^+$, we let
    \begin{enumerate}[(i)]
    \item $V_{\langle 1\rangle}\triangleq 0$. We call it the $1^{st}$ $(g,h)$-extended space with $z$;
    \item $H_{\langle 1\rangle}\triangleq H$. We call it the  $1^{st}$  $(g,h)$-extended coalgebra with $z$ (of $H$).
    \end{enumerate}
\item For $z\in\left({}^g H_2{}^h\right)^+$, as in Proposition \ref{11}, we write
$$
\Delta(z)= g\otimes z + \sum\limits_{j\in J_1} \sum\limits_{k\in G(H)} x_{1,j}^{(k)}\otimes y_{1,j}^{(k)}+ z\otimes h ,
$$
where $J_1$ is a finite set, and each $x_{1,j}^{(k)}\in\left({}^g H_1{}^k\right)^+ ,y_{1,j}^{(k)}\in \left({}^k H_1{}^h\right)^+$ for all $j\in J_1$ and $k\in G(H)$.

Then we let
    \begin{enumerate}[(i)]
    \item $V_{\langle 2\rangle}$ be a vector space with basis $\{z_{1,j}^{(k)} \mid j\in J_1,k\in G(H)\}$.
        We call it the $2^{nd}$ $(g,h)$-extended space with $z$;
    \item $H_{\langle 2\rangle}\triangleq H\oplus V_{\langle 2\rangle}$, on which the coalgebra structure is defined by
        $\Delta|_H=\Delta_H$, $\varepsilon|_H=\varepsilon_H$, and
        $$
        \left\{
          \begin{array}{l}
            \Delta(z_{1,j}^{(k)})= g\otimes z_{1,j}^{(k)} + x_{1,j}^{(k)}\otimes y_{1,j}^{(k)}+ z_{1,j}^{(k)}\otimes h   \\
            \varepsilon(z_{1,j}^{(k)})=0
        \end{array}
        \right.
        (\forall j\in J_1,k\in G(H)).
        $$
        We call $H_{\langle 2\rangle}$ the $2^{nd}$  $(g,h)$-extended coalgebra with $z$.
    \end{enumerate}
\end{enumerate}
\end{definition}

Some straightforward property of $H_{\langle 2\rangle}$ should be noted.

\begin{proposition}\label{16}
For any $g,h\in G(H)$ and $z\in \left({}^g H_2{}^h\right)^+$, we let $V_{\langle 2\rangle}$ and $H_{\langle 2\rangle}=H\oplus V_{\langle 2\rangle}$ be the $2^{nd}$ $(g,h)$-extended space and coalgebra with $z$ respectively. Then
\begin{enumerate}
\item $(H_{\langle 2\rangle},\Delta,\varepsilon)$ is  a coalgebra containing $H$, with its coradical being $H_0$;
\item The finite set $\{z_{1,j}^{(k)} \mid j\in J_1,k\in G(H)\}\subseteq V_{\langle 2\rangle}$ consists of elements, each of which is the $(1,3)$-entry of a $(g,h)$-upper triangular multiplicative matrix (over $H_{\langle 2\rangle}$) of order $3$, and they satisfy
    $$
    z-\sum\limits_{j\in J_1,k\in G(H)} z_{1,j}^{(k)}\in P_{g,h}(H_{{\langle 2\rangle}}).
    $$
\end{enumerate}
\end{proposition}

\begin{proof}\hspace{1ex}\\
\begin{enumerate}
\item
The definition of $(H_{\langle 2\rangle},\Delta,\varepsilon)$ induces multiplicative matrices
$$
\mathcal{Z}_{1,j}^{(k)}
=\left(
\begin{array}{ccc}
  g  &  x_{1,j}^{(k)}  &  z_{1,j}^{(k)}  \\
  0  &  k  &  y_{1,j}^{(k)}  \\
  0  &  0  & h
\end{array}
\right)
\in\mathcal{M}_3(H_{\langle 2\rangle})
$$
for all $j\in J_1$ and $k\in G(H)$, this implies that
$$
\left\{
  \begin{array}{ll}
    (\Delta\otimes id)\circ\Delta(\mathcal{Z}_{1,j}^{(k)})
    =\mathcal{Z}_{1,j}^{(k)}~\widetilde{\otimes}~\mathcal{Z}_{1,j}^{(k)}~\widetilde{\otimes}~\mathcal{Z}_{1,j}^{(k)}
    =(id\otimes\Delta)\circ\Delta(\mathcal{Z}_{1,j}^{(k)})   \\
    (\varepsilon\otimes id)\circ\Delta(\mathcal{Z}_{1,j}^{(k)})
    =\mathcal{Z}_{1,j}^{(k)}
    =(id\otimes\varepsilon)\circ\Delta(\mathcal{Z}_{1,j}^{(k)}).
  \end{array}
\right. .
$$
So
$$
\left\{
  \begin{array}{ll}
    (\Delta\otimes id)\circ\Delta(z_{1,j}^{(k)})
    =(id\otimes\Delta)\circ\Delta(z_{1,j}^{(k)})   \\
    (\varepsilon\otimes id)\circ\Delta(z_{1,j}^{(k)})
    =z_{1,j}^{(k)}
    =(id\otimes\varepsilon)\circ\Delta(z_{1,j}^{(k)}),
  \end{array}
\right.
$$ $\forall j\in J_1,k\in G(H)$,
which deduces that the coassociative axiom and counit axiom on $(H_{\langle 2\rangle},\Delta,\varepsilon)$ hold.

It is apparent that $(H_{\langle 2\rangle})_0=H_0$.
\item The equation holds by a straightforward calculation.
\end{enumerate}
\end{proof}

Next we define inductively the $n^{th}$ $(g,h)$-extended spaces and coalgebras with $z\in\left({}^g H_n{}^h\right)^+$ for $n\geq 3$.

\begin{definition}\label{20}
Let $n\geq 3$. Let $(H,\Delta_H,\varepsilon_H)$ be a finite-dimensional pointed coalgebra. For any $g,h\in G(H)$, and $z\in\left({}^g H_n{}^h\right)^+$, as in Proposition \ref{11}, we write
$$
\Delta(z)
= g\otimes z + \sum\limits_{i=1}^{n-1} \sum\limits_{j\in J_i} \sum\limits_{k\in G(H)} x_{i,j}^{(k)}\otimes y_{i,j}^{(k)}+ z\otimes h ,
$$
where each $J_i$ is finite and each $x_{i,j}^{(k)}\in\left({}^g H_i{}^k\right)^+ ,y_{i,j}^{(k)}\in \left({}^k H_{n-i}{}^h\right)^+$ for all $1\leq i\leq n-1$, $j\in J_i$ and $k\in G(H)$. Assume that following spaces have been defined already:
\begin{enumerate}
\item $U_{i,j}^{(k)}$ and $G_{i,j}^{(k)}=H\oplus U_{i,j}^{(k)}$ are the $i^{th}$ $(g,k)$-extended space and coalgebra with $x_{i,j}^{(k)}$ respectively. There exists a finite set $\textswab{X}_{i,j}^{(k)}\subseteq U_{i,j}^{(k)}$ consisting of elements, each of which is the $(1,i+1)$-entry of a $(g,k)$-upper triangular multiplicative matrix (over $G_{i,j}^{(k)}$) of order $i+1$, such that
    $$
    x_{i,j}^{(k)}-\sum\limits_{x\in\textswab{X}_{i,j}^{(k)}} x\in P_{g,k}(G_{i,j}^{(k)});
    $$
\item $V_{i,j}^{(k)}$ and $H_{i,j}^{(k)}=H\oplus V_{i,j}^{(k)}$ are the $(n-i)^{th}$ $(k,h)$-extended space and coalgebra with $y_{i,j}^{(k)}$ respectively. There exists a finite set $\textswab{Y}_{i,j}^{(k)}\subseteq V_{i,j}^{(k)}$ consisting of elements, each of which is the $(1,n-i+1)$-entry of a $(k,h)$-upper triangular multiplicative matrix (over $H_{i,j}^{(k)}$) of order $n-i+1$, such that
    $$
    y_{i,j}^{(k)}-\sum\limits_{y\in\textswab{Y}_{i,j}^{(k)}} y\in P_{k,h}(H_{i,j}^{(k)}).
    $$
\end{enumerate}

Then define:
    \begin{enumerate}[(i)]
    \item
        $$
        V_{\langle n\rangle}=\left[\bigoplus\limits_{1\leq i\leq n-1,j\in J_i,k\in G(H)} \left(U_{i,j}^{(k)}\oplus
        V_{i,j}^{(k)}\right)\right]\oplus\left[\bigoplus\limits_{1\leq i\leq n-1,j\in J_i,k\in
        G(H)}{\Bbbk}(\textswab{X}_{i,j}^{(k)}\times\textswab{Y}_{i,j}^{(k)})\right] .
        $$
        We call it the $n^{th}$ $(g,h)$-extended space with $z$;
    \item $H_{\langle n\rangle}\triangleq H\oplus V_{\langle n\rangle}$, on which the coalgebra structure is defined by
        $\Delta|_{\overline{H}}=\Delta_{\overline{H}}$, $\varepsilon|_{\overline{H}}=\varepsilon_{\overline{H}}$, and
        $$
        \left\{
        \begin{array}{l}
            \Delta(z_{x,y})= g\otimes z_{x,y} + x\otimes y+ z_{x,y}\otimes h   \\
            \varepsilon(z_{x,y})=0
        \end{array}
        \right.
        ~\left(
        \begin{array}{c}
            \forall 1\leq i\leq n-1,j\in J_i,k\in G(H),  \\
            \forall z_{x,y}\triangleq (x,y)\in\textswab{X}_{i,j}^{(k)}\times\textswab{Y}_{i,j}^{(k)}
        \end{array}
        \right),
        $$
        where
        $$
        \overline{H}\triangleq H\oplus\left[\bigoplus\limits_{1\leq i\leq n-1,j\in J_i,k\in G(H)} \left(U_{i,j}^{(k)}\oplus V_{i,j}^{(k)}\right)\right]
        $$
        is constructed because $H$ is the subcoalgebra of each $G_{i,j}^{(k)}$ and $H_{i,j}^{(k)}$ by Lemma \ref{18}.
        We call $H_{\langle n\rangle}$ the $n^{th}$ $(g,h)$-extended coalgebra with $z$.
    \end{enumerate}
\end{definition}

The $n^{th }$ $(g,h)$-extended coalgebra has similar properties to the $2^{nd}$ $(g,h)$-extended coalgebra.

\begin{lemma}\label{13}
For any positive integer $n\geq 2$, $g,h\in G(H)$ and any $z\in\left({}^g H_n{}^h\right)^+$, let $V_{\langle n\rangle}$ and $H_{\langle n\rangle}=H\oplus V_{\langle n\rangle}$ be the $n^{th}$ $(g,h)$-extended space and coalgebra with $z$ respectively. Then
\begin{enumerate}
\item $(H_{\langle n\rangle},\Delta,\varepsilon)$ is indeed a coalgebra containing $H$, and its coradical is $H_0$;
\item $z-\sum\limits_{1\leq i\leq n-1,j\in J_i,k\in G(H)}\sum\limits_{x\in\textswab{X}_{i,j}^{(k)},y\in\textswab{Y}_{i,j}^{(k)}} z_{x,y}\in P_{g,h}(H_{\langle n\rangle})$.
\item There exists a finite set $\textswab{Z}\subseteq H_{\langle n\rangle}$ such that every element in $\textswab{Z}$
    is the $(1,n+1)$-entry of a $(g,h)$-upper triangular multiplicative matrix (over $H_{\langle n\rangle}$) of order $n+1$, and the sum of them equals to
    $z$.
\end{enumerate}
\end{lemma}

\begin{proof}
Our first goal is to show that every element in the finite set
$$
\textswab{Z}^\prime \triangleq\bigcup\limits_{1\leq i\leq n-1,j\in J_i,k\in G(H)}
  \textswab{X}_{i,j}^{(k)}\times\textswab{Y}_{i,j}^{(k)}\subseteq H_{\langle n\rangle}
$$
is the $(1,n+1)$-entry of a $(g,h)$-upper triangular multiplicative matrix of order $n+1$. By Definition \ref{20}, $\textswab{X}_{i,j}^{(k)}$ and $\textswab{Y}_{i,j}^{(k)}$ consist of elements from the top-right corner of multiplicative matrices. Thus for each $1\leq i\leq n-1$, $j\in J_i$ and $k\in G(H)$, and for each $x\in\textswab{X}_{i,j}^{(k)}$ and $y\in\textswab{Y}_{i,j}^{(k)}$, there are upper triangular matrices
$$
\left(
  \begin{array}{ccc}
    g      & \cdots & x       \\
    \vdots & \ddots & \vdots  \\
    0      & \cdots & k       \\
  \end{array}
\right) \in\mathcal{M}_{i+1}(H_{\langle n\rangle})
,~
\left(
  \begin{array}{ccc}
    k      & \cdots & y       \\
    \vdots & \ddots & \vdots  \\
    0      & \cdots & h       \\
  \end{array}
\right) \in\mathcal{M}_{n-i+1}(H_{\langle n\rangle}).
$$
Note that such two matrices can be put into a $(g,h)$-upper triangular multiplicative matrix
$$
\mathcal{Z}_{x,y}=
\left(
  \begin{array}{ccccc}
    g      & \cdots & x      & \cdots & z_{x,y} \\
    \vdots & \ddots & \vdots & 0      & \vdots  \\
    0      & \cdots & k      & \cdots & y       \\
    \vdots & 0      & \vdots & \ddots & \vdots  \\
    0      & \cdots & 0      & \cdots & h       \\
  \end{array}
\right)\in\mathcal{M}_{n+1}(H_{\langle n\rangle})
~~(x\in \textswab{X}_{i,j}^{(k)},y\in \textswab{Y}_{i,j}^{(k)}).
$$
Hence the statement about $\textswab{Z}^\prime$ is proved.

\begin{enumerate}
\item
Recall in Definition \ref{20} that
$$
H_{\langle n\rangle}\triangleq\overline{H}
\oplus\left[ \bigoplus\limits_{1\leq i\leq n-1,j\in J_i,k\in G(H)}
{\Bbbk}(\textswab{X}_{i,j}^{(k)}\times\textswab{Y}_{i,j}^{(k)})\right] ,
$$
where
$$
\overline{H}\triangleq H\oplus\left[\bigoplus\limits_{1\leq i\leq n-1,j\in J_i,k\in G(H)} \left(U_{i,j}^{(k)}\oplus V_{i,j}^{(k)}\right)\right]
$$
is already a coalgebra. Thus in order to prove $H_{\langle n\rangle}$ is a coalgebra, it is enough to show that coassociative axiom and counit axiom holds on every element in
$
\textswab{Z}'\triangleq\bigcup\limits_{1\leq i\leq n-1,j\in J_i,k\in G(H)}
 \textswab{X}_{i,j}^{(k)}\times\textswab{Y}_{i,j}^{(k)}
$.
The proof is similar to the one of Proposition \ref{16} (1), since every element in the finite set $\textswab{Z}^\prime$ is the $(1,n+1)$-entry of a $(g,h)$-upper triangular multiplicative matrix of order $n+1$.

It is apparent that $(H_{\langle n\rangle})_0=H_0$.

\item
Straightforward by calculations. Let
$$
z^\prime \triangleq\sum\limits_{1\leq i\leq n-1,j\in J_i,k\in G(H)}\sum\limits_{x\in\textswab{X}_{i,j}^{(k)},y\in\textswab{Y}_{i,j}^{(k)}}
z_{x,y} \in H_{\langle n\rangle}
$$
denote the sum of all elements in $\textswab{Z}^\prime$. By Definition \ref{20} that
$$
\Delta(z_{x,y})= g\otimes z_{x,y} + x\otimes y+ z_{x,y}\otimes h,
~\left(
\begin{array}{c}
    \forall 1\leq i\leq n-1,j\in J_i,k\in G(H),  \\
    \forall z_{x,y}\triangleq (x,y)\in\textswab{X}_{i,j}^{(k)}\times\textswab{Y}_{i,j}^{(k)}
\end{array}
\right),
$$
we have
$$
\Delta(z^\prime)=g\otimes z^\prime + \sum\limits_{i=1}^{n-1} \sum\limits_{j\in J_i} \sum\limits_{k\in G(H)} x_{i,j}^{(k)}\otimes y_{i,j}^{(k)}+ z^\prime\otimes h.
$$
So $z-z^\prime \in P_{g,h}(H_{\langle n\rangle})$.

\item
Choose one tuple $(i_0,j_0,k_0)$, satisfying $1\leq i\leq n-1,j\in J_i,k\in G(H)$ and an element $w\in\textswab{X}_{i_0,j_0}^{(k_0)}\times\textswab{Y}_{i_0,j_0}^{(k_0)}$. We claim that the element $z-(z^\prime -w)$ is also the $(1,n+1)$-entry of a $(g,h)$-upper triangular multiplicative matrix over $H_{\langle n\rangle}$ of order $n$.

Assume that $w$ is the $(1,n+1)$-entry of a $(g,h)$-upper triangular multiplicative matrix $\mathcal{W}\triangleq(w_{i\,j})\in\mathcal{M}_{n+1}(H_{\langle n\rangle})$. It follows that $\Delta(w)=g\otimes w+w\otimes h +\sum\limits_{l=1}^n w_{1\,l}\otimes w_{l\,n+1}$. Thus
\begin{eqnarray*}
&&\Delta(z-(z^\prime -w))=\Delta(z-z')+\Delta(w)  \\
&=&g\otimes(z-z^\prime)+(z-z^\prime)\otimes h+g\otimes w+w\otimes h +\sum\limits_{l=1}^n w_{1\,l}\otimes w_{l\,n+1}  \\
&=&g\otimes[z-(z^\prime -w)]+\sum\limits_{l=1}^n w_{1\,l}\otimes w_{l\,n+1}+[z-(z^\prime -w)]\otimes h,
\end{eqnarray*}
and $\varepsilon(z-(z^\prime -w))=0$. That is to say $\mathcal{W}$ will be still multiplicative of if $z-(z^\prime -w)$ is substituted for its $(1,n+1)$-entry $w_{n+1\,n+1}$. Hence $z-(z^\prime -w)$ is also the $(1,n+1)$-entry of a $(g,h)$-upper triangular multiplicative matrix.

Finally set $\textswab{Z}\triangleq(\textswab{Z}^\prime \cup\{z-(z^\prime -w)\})\setminus \{w\}$. Apparently every element of $\textswab{Z}$ is the $(1,n+1)$-entry of a $(g,h)$-upper triangular multiplicative matrix of order $n+1$, and $z=(z^\prime -w)+[z-(z^\prime -w)]$ is the sum of all the elements of $\textswab{Z}$. This completes the proof.
\end{enumerate}
\end{proof}

By Lemma \ref{13}, every $z\in\left({}^g H_n{}^h\right)^+$ is a finite sum of elements in $H_{\langle n\rangle}$, and each of them is the $(1,n+1)$-entry of a $(g,h)$-upper triangular multiplicative matrix. With the help of it and Lemma \ref{17} , we can finally prove the theorem in this section.

Since the coradical of a pointed Hopf algebra $H$ is a group algebra, it has finite exponent dividing the dimension of $H$ (Kashina \cite[Section 1]{Kashina 1999}).

\begin{theorem}
Let $H$ be a finite-dimensional pointed Hopf algebra over a field of characteristic $p>0$. Assume $H=H_n$ for some positive integer $n$. Denote $d=\exp(H_0)<\infty$. Then $\exp(H)\leq d p^{ \lfloor \log_p n \rfloor +1 }$.
\end{theorem}

\begin{proof}
By Proposition \ref{2} (3), $H=H_n=H_0\oplus\left[ \bigoplus\limits_{g,h\in G(H)} \left({}^g H_n{}^h\right)^+ \right]$. Thus it is enough to show that for any $g,h\in G(H)$ and any $z\in\left({}^g H_n{}^h\right)^+$, $z^{\left[ d p^{ \lfloor \log_p n \rfloor +1 }\right]}=0=\varepsilon(z)$ holds.

Let $H_{\langle n\rangle}$ be the $n^{th}$ $(g,h)$-extended coalgebra with $z$. Since $H$ is a subcoalgebra of $H_{\langle n\rangle}$, by Lemma \ref{17} (2), we have the following commutative diagram:
$$
{
\xymatrix{
H  \ar@{^{(}->}[d]_-{\iota_{H_{\langle n\rangle}}|_H}  \ar[r]^-\varphi  &  T(H)/({\rm Ker}~\pi_H)     \ar@{^{(}->}[d]^-\phi  \\
T(H_{\langle n\rangle})  \ar[r]_-\theta   &  T(H_{\langle n\rangle})/[T(H_{\langle n\rangle})({\rm Ker}~\pi_H)T(H_{\langle n\rangle})]   }
},
$$
where $\varphi$ is a bialgebra isomorphism, $\phi$ is a bialgebra monomorphism, $\theta$ is a quotient bialgebra map, and $\iota_{H_{\langle n\rangle}}|_H$ is just a coalgebra inclusion. Recall that $\pi_H:T(H)\rightarrow H, h_1\otimes h_2\otimes\cdots\otimes h_l\mapsto h_1h_2\cdots h_l$ in Lemma \ref{17} (1), where $l\in{\mathbb{ Z}}^+$ and $h_1,h_2,\cdots, h_l\in H$.

In order to show $z^{\left[ d p^{ \lfloor \log_p n \rfloor +1 } \right]}=0=\varepsilon(z)=0$ in $H$, we only need to prove $(\phi\circ\varphi)(z^{\left[ d p^{ \lfloor \log_p n \rfloor +1 } \right]})=0$ in $$T(H_{\langle n\rangle})/[T(H_{\langle n\rangle})({\rm Ker}~\pi_H)T(H_{\langle n\rangle})],$$ since $\phi\circ\varphi$ is injective.

On the other hand, we have
$$
(\phi\circ\varphi)(z^{\left[ d p^{ \lfloor \log_p n \rfloor +1 } \right]})
= (\phi\circ\varphi)(z)^{\left[ d p^{ \lfloor \log_p n \rfloor +1 } \right]}
= [\theta\circ\iota_{H_{\langle n\rangle}}(z)]^{\left[ d p^{ \lfloor \log_p n \rfloor +1 } \right]}
,
$$
because $\phi$, $\varphi$ and $\theta$ are bialgebra maps and the diagram above commutes. Note that $\iota_{H_{\langle n\rangle}}(z)$ is a finite sum of elements, each of which is the $(1,n+1)$-entry of a $(g,h)$-upper triangular multiplicative matrix $\mathcal{Z}$ over $T(H_{\langle n\rangle})$. And so is $\theta\circ\iota_{H_{\langle n\rangle}}(z)$, where the corresponding matrix is $\theta(\mathcal{Z})$, since $\theta$ is a coalgebra map.

It remains to show that each $\mathcal{Z}$ has order dividing $d p^{ \lfloor \log_p n \rfloor +1 }$. By Proposition \ref{10}, it is enough to verify that every element on the diagonal of $\theta(\mathcal{Z})$ has order dividing $d$ in $$T(H_{\langle n\rangle})/[T(H_{\langle n\rangle})({\rm Ker}~\pi_H)T(H_{\langle n\rangle})]. $$ This is because for any $g\in G(H)=G(H_{\langle n\rangle})$, we have
$$
(\theta\circ\iota_{H_{\langle n\rangle}})(g)^d=(\phi\circ\varphi)(g)^d=(\phi\circ\varphi)(g^d)=1.
$$
The proof of the theorem is now complete.
\end{proof}

\section*{Funding}

This work was supported by NNSF of China (No. 11331006).

\end{document}